\documentclass{ifacconf}

\usepackage{natbib}        
\usepackage{amsmath,amssymb,latexsym}
\usepackage{array,graphicx,dsfont}
\usepackage{stmaryrd,xcolor,psfrag} 
\usepackage{epstopdf}
\usepackage{subfigure}

\newcommand\ds{\displaystyle}

\newcommand\vect{\textnormal{Vect}}

\newtheorem{Rk}{Remark}
\newtheorem{Pp}{Property}
\newtheorem{Lem}{Lemma}
\newtheorem{definition}{Definition}

\newenvironment{proof}{{\it Proof :~}}{\hfill$\square$\\}

\begin{document}
\begin{frontmatter}

\title{
Stability analysis of a system coupled to a heat equation.\thanksref{footnoteinfo}
} 

\today
\thanks[footnoteinfo]{This work is supported by the ANR project SCIDiS contract number 15-CE23-0014.}

\author[First]{Lucie Baudouin}, 
\author[First]{Alexandre Seuret},
\author[First]{Frederic Gouaisbaut} 

\address[First]{LAAS-CNRS, Universit\'e de Toulouse, CNRS, UPS, Toulouse, France.
}

\begin{abstract}            
As a first approach to the study of systems coupling finite and infinite dimensional natures, this article addresses the stability of
a system of ordinary differential equations coupled with a classic heat equation using a Lyapunov functional technique. 
Inspired from recent developments in the area of time delay systems, a new methodology to study the stability of such a class of distributed parameter systems is presented here. 
The idea is to use a polynomial approximation of the infinite dimensional state of the heat equation in order to build an enriched energy functional.
A well known efficient integral inequality (Bessel inequality) will allow to obtain stability conditions expressed in terms of linear matrix inequalities. 
We will eventually test our approach on academic examples in order to illustrate the efficiency of our theoretical results. 
\end{abstract}

\end{frontmatter}

 \section{Introduction }\label{intro}
 
 Coupling  a classical finite dimensional system to a partial differential equation (PDE) presents not only interesting theoretical challenges but can also formalize various applicative situations.
Effectively, as the solution of the PDE is a state belonging to an infinite dimensional functional space, its coupling with a finite dimensional system brings naturally new difficulties in stability study and/or control of the coupled system. 
\newline
We can also list several specific situations worth being modeled by this king of heterogeneous coupled system,  see e.g. \cite{HAK-Autom16} or \cite{Aamo-TAC13}. For example, the finite dimensional systems could represent a dynamic controller for a system modeled by a PDE (see \cite{andrea1994}, \cite{krstic2009delay} and references therein). Instead, a system of ordinary differential equations (ODEs) can model a component coupled to a phenomenon described by PDEs as in \cite{DTV-SCL14}. Conversely, the PDE can model an  actuator or sensor's behavior and the goal could be to study the stabilization of a finite dimensional system in spite of the introduction of the actuator/sensor's dynamics.
\newline
Actually, the last decade has seen the emergence of number of papers concerning the stability or control of this type of coupled systems (as in \cite{SK-JFI10}, \cite{krstic2009delay}, see also  references therein). 
When considering such a coupling of equations of different nature, it is important to highlight that the notion of stability regarding PDEs is not as generic as for classical systems of ODEs. It depends specifically on the type of PDE under consideration, on the functional space where the solution belongs and the choice of an appropriate norm (in other words the definition of the energy of the infinite dimensional state), see \cite{TX-SCL11}. Of course, the type of interconnection between the ODE and the PDE and the boundary conditions of the PDE also plays a role (see for instance the reference book \cite{CurtainBook} or \cite{Bastin-Coron-Book} for a rather complete exposition of the stability and stabilization problem). 
\newline
One classical way to study the stability of such a coupled system relies on discretization techniques leading to some finite dimensional systems to be studied. The question of convergence (from the discretized to the corresponding continous system) of the results is then quite natural and may be complicated to deal with (see \cite{Morris1994}). That's the reason why several researches have turned to direct approaches: the objective is to determine a Lyapunov functional for the overall system directly,  without going through a discretization scheme. This gave rise to many interesting methodologies. Hence, a first one relies on the semi-group theory to model the overall system and it may lead, as in \cite{fridman2014introduction}, to some Linear Operator Inequalities to be solved numerically. Unfortunately, this approach remains quite limited (see \cite{fridman2009exponential}) and works finally only for small dimensional ODE systems since no numerical tools are available to solve these Linear Operator Inequalities. Furthermore, the generic semi-group approach generally fails to develop a constructive approach for the design of Lyapunov functionals.

Another possible approach considers the design of a Lyapunov functional which is usually based on the sum of a classical Lyapunov functional identified for each part of the system under consideration. When dealing with the PDE of a coupled system, its Lyapunov functional is actually the ``energy" of the PDE (see \cite{Prieur2012}, \cite{Papachristodoulou2006}). In the book \cite{krstic2009delay}, chapter 15, the control of a finite dimensional system connected to an actuator/sensor modeled by a heat equation with Neumann and Dirichlet boundary conditions is considered. The author adopts the backstepping method employed originally in the case of the transport (or delay) equation. The resulting feedback system is equivalent to a finite dimensional exponentially stable system cascaded with a heat condition. The choice of an appropriate Lyapunov function as a sum of the energy of the heat equation and a classical quadratic function for the finite dimensional system allows to prove the exponential stability of the overall system. 
Recently, several approaches based on an optimisation procedure have been developped. Starting from a semi-group modeling of the PDEs, the authors of  \cite{Peet-TAC17} construct a very general Lyapunov functional which parameters are optimised via a sum of square procedure (see also \cite{AVP-Autom16}).  This methodology is then applied to the controller or observer design. 
\newline
Considering in this article a situation where a heat type phenomenon is to be controlled at its boundary by a finite dimensional dynamical controller, we are interested with the efficient and numerically tractable stability analysis of the closed-loop system. More precisely, we will not work on the control design, but on the stability study of a system coupling a one-dimensional heat equation and an ODE.
The practical interest of such a model is reflected for example in the study of temperature control systems using a thermocouple as a heat sensor.
\newline
Our task in this article will be to study a finite dimensional ODE system coupled with a heat equation in 1-d in space, where the interconnection is performed through the boundary of the space domain. 
We aim at proving exponential stability results, meaning that starting from an arbitrary initial condition, the whole system's solution follows a time  trajectory that exponentially converges in spatial norm to an equilibrium state. Nevertheless, the stability analysis is challenging since it depends strongly on the norm chosen to measure the deviation with respect to the steady state (for the PDE part specifically). 
But above all, our goal is to provide practical stability tests for the whole system that can take into account both the finite dimensional state and its interplay with the infinite dimensional state of the PDE. 
It will be performed thanks to the construction of a general Lyapunov functional based on the weighed classical energy of the full system enriched by a quadratic term built on a truncation of the distributed state. 
To this end, we will use the projection of the state over a set of polynomials and take advantage of this approximation to provide tractable stability conditions for the whole coupled system. A first step of our study, using only the mean value of the PDE state as a rough approximation, was presented in the conference paper \cite{BSG-IFAC17}. Notice that the tools of this approach have also been used in \cite{barreau2017lyapunov} in order to study the stability of a coupling between an ODE and a hyperbolic equation.

{\em Notation.}  
As usual, $\mathbb N$ denote the sets of positive integers,
$\mathbb R^+$, $\mathbb R^{n}$, and $\mathbb R^{n\times m}$ the positive
reals, $n$-dimensional vectors and $n\times m$ matrices ; the Euclidean norm writes $|\cdot|$. For any matrix $P$ in $\mathbb R^{n\times n}$, we denote $\mbox{He}(P)=P+P^\top $  (where $P^\top$ is the transpose matrix) and $P\succ0$ means that $P$ is symmetric positive definite, \textit{ie} $P\in\mathbb S^{n}_+$. For a partitioned matrix, the symbol ${\ast}$ stands for symmetric blocks and $I$ is the identity, $0$ the zero matrix.  
The partial derivative on a function $u$ with respect to $x$ is denoted $\partial_x u = \frac{\partial u}{\partial x}$ (while the time derivative of $X$ is $\dot X = \frac{dX}{dt}$). Finally, using $L^2(0,1)$ for the Hilbert space of square integrable functions, one writes $\|z\|^2 = \int_{0}^1 |z(x)|^2  \,dx = \left<z,z\right>$, and we also define the Sobolev spaces $H^1(0,1) = \{ z\in L^2(0,1), \partial_x z \in L^2(0,1) \}$ and its norm by $\|z\|^2_{H^1(0,1)} = \|z\|^2 + \|\partial_x z\|^2$, $H^2(0,1) = \{ z\in L^2(0,1), \partial_x z \in L^2(0,1), \partial_{xx} z \in L^2(0,1) \}$ and its norm by $\|z\|^2_{H^2(0,1)} = \|z\|^2 + \|\partial_x z\|^2+\|\partial_{xx} z\|^2$.

{\em Outline.}
A thorough description of the system under study will be given in Section~\ref{Problem}. Then, Section~\ref{Tools} will detail the main tools of the proof of the stability result presented in Section~\ref{Stability}. An illustrating example of this theoretical result will conclude in Section~\ref{Numerics}.

 \section{Problem Description}\label{Problem}
 \subsection{A coupled system}\label{coupledsystem}
Consider the coupling of a finite dimensional system in the variable $X \in \mathbb R^n$ with a heat partial differential equation in the scalar variable $u$, in the following  way:
\begin{equation}\label{ODEheat}
		\left\{\begin{array}{ll}
		\dot X (t) = A X(t) + Bu(1,t) & \qquad t >0,\\
		\partial_t u (x,t) = \gamma \partial_{xx} u(x,t) ,& \qquad x \in (0,1), t >0,\\
		u (0,t) = C X(t),&\qquad t >0\\
		\partial_x u (1,t) = 0,&\qquad t >0.
		\end{array}\right.
\end{equation}
 
 The state vector of the system is the pair $(X(t),u(x,t)) \in \mathbb R^n \times\mathbb R $ and it satisfies the compatible initial datum 
$(X(0),u(x,0)) = (X^0,u^0(x))$ for $x\in(0,1)$. The thermal diffusivity is denoted $\gamma \in\mathbb R_+$ and the matrix $A \in\mathbb R^{n\times n}$, the vectors $B\in \mathbb R^{n\times 1}$ and $C \in \mathbb R^{1\times n}$ are constant.

\begin{Rk}
One can imagine different situations that can be translated into the coupled system \eqref{ODEheat}. As a toy problem of more complicated situations, the system we study already allows to face several difficulties inherent to a situation mixing finite and infinite dimensional states. Nevertheless, we can describe two more physical situations that could be simplified as our toy problem : either a finite dimensional system confronted with a thermocouple sensor, or a heat device connected to a  finite dimension dynamic controller. Anyway, these are only mere ideas that could link ODEs with a heat PDE and we remain here at a simplified but still challenging level.
\end{Rk}

 \subsection{Existence and regularity of the solutions}\label{coupledsystem}

Before anything else, one should know that the partial differential equation $\partial_t u - \gamma \partial_{xx}  u = 0$ in \eqref{ODEheat} of unknown $u=u(x,t)$ is a classic heat PDE and if the boundary data are of Dirichlet homogeneous type (\textit{i.e.} $u (0,t) = u(1,t) = 0$) and the initial datum $u(\cdot,0) = u^0$ belongs to $H_0^1(0,1)$, it has a unique solution $u$ satisfying (see \cite{Brezis})
$$
\begin{aligned}
	&u \in C([0,+\infty[;H_0^1(0,1))\cap L^2(0, + \infty;H^2(0,1)), \\
&\partial_t u \in L^2(0,+\infty ;L^2(0,1 )).
\end{aligned}
$$
In this article, we are dealing with System \eqref{ODEheat}, coupling ODEs with a heat equation through its boundary data, system for which we should start with the existence and regularity of the solution $(X,u)$. A Galerkin method (see e.g. \cite{EvansBook}) is the key of the proof of such a result, stated in the following lemma.
\begin{Lem}
Assuming that the initial data $(X^0, u^0)$ belong to $\mathbb R^n \times L^2(0,1)$, system \eqref{ODEheat} admits a unique solution $(X,u)$ such that 
\begin{multline*}
X \in C([0,+ \infty[;\mathbb R^n) \\
u\in C([0,+\infty[;H^1(0,1))\cap L^2(0, + \infty;H^2(0,1))  \\
\textnormal{ and } \partial_t u \in L^2(0,+\infty ;L^2(0,1 )).
\end{multline*}
 \end{Lem}
 \vspace{-0.5cm}
In the sake of consistency with the Lyapunov approach we will use in the stability study of our coupled system, we give below the formal proof using the Galerkin energy based method. One should notice that the proposed approach has been developed  in particular for parabolic equations in \cite[Chapter 7.1]{EvansBook}.\\
\begin{proof}
Let us define the total energy of System~\eqref{ODEheat} by
$
E(X(t),u(t)) = |X(t)|_n^2 + \|u(t)\|^2_{H^1(0,1)}.
$
In the sequel, we will write $E(t)=E(X(t),u(t))$ in order to simplify the notations.
Easy calculations based on the equations of System~\eqref{ODEheat} and integrations by parts give :
\begin{eqnarray*}
\dot E(t)~ &=& ~X(t)^\top (A^\top + A) X(t) + 2u(1,t)B^\top X(t) \\
&&-~2 \gamma \|\partial_x u(t) \|^2- 2 \partial_xu(0,t)(\gamma C + CA)X(t) \\
&&- ~2 \partial_xu(0,t)CBu(1,t) -2 \gamma \|\partial_{xx} u(t) \|^2.
\end{eqnarray*}
First, in order to deal with the three cross terms mixing $X(t)$, $u(1,t)$ and $\partial_x u(0,t)$, we use Young's inequality $\left(ab\leq \frac{a^2}{2\varepsilon} + \frac{\varepsilon b^2}{2} \right)$, and choosing each time appropriately the tuning parameter $\varepsilon$, one can obtain 
\begin{multline*}
\dot E(t) \leq M |X(t)|_n^2 + M |u(1,t)|^2 -2 \gamma \|\partial_x u(t) \|^2  \\
+ ~2 \epsilon |\partial_xu(0,t)|^2  -2 \gamma \|\partial_{xx} u(t) \|^2,
\end{multline*}
where, from now on,  $M>0$ is a generic contant depending on $A,B,C,\gamma,\epsilon$. 
\\
Second, since we have the Sobolev embeddings $H^1(0,1) \subset C([0,1])$ and $H^2(0,1) \subset C^1([0,1])$, one can write, omitting the time variable $t$, that
$$
 |u(1)|^2 \leq 2 |u(0)|^2  + 2 \|\partial_{x} u \|^2 \text{ and } |\partial_xu(0)|^2 \leq \|\partial_{xx} u \|^2.
$$
Along with $u(0) = CX$, it leads to  
\begin{eqnarray*}
\dot E(t) &\leq& M |X(t)|_n^2 + M \|\partial_x u(t) \|^2  + ~2 (\varepsilon - \gamma) \|\partial_{xx} u(t) \|^2.
\end{eqnarray*}
On the one hand, choosing $0< \epsilon < \gamma$, we can eliminate the last term (in $\|\partial_{xx} u(t) \|^2$, because its coefficient is negative) so that we get
$\dot E(t) \leq M E(t)$ ensuring, from Gr\"onwall's inequality, the existence of a unique solution $(X,u)$ in the space  
$ C([0,+\infty[;\mathbb R^n \times H^1(0,1 )).$ \\
On the other hand, we can also move the $\|\partial_{xx} u(t) \|^2$ term to the left hand side of the estimate and deduce from the existence of a finite upper bound that $u \in L^2(0, + \infty;H^2(0,1 ))$. Thereafter, using the heat equation $\partial_t u = \gamma \partial_{xx} u$ from \eqref{ODEheat}, we also get
$
\partial_t u \in L^2(0,+\infty ;L^2(0,1 )).
$
\end{proof}\\
These somewhat terse explanations allows us to manipulate the solution $(X,u)$ in the appropriate functional space along this article.

 \subsection{Equilibrium and stability}\label{Equilibrium}
 
 As proved in the preliminary study \cite{BSG-IFAC17}, if the matrix $A+BC$ is non singular, then system \eqref{ODEheat} has a unique equilibrium $(X_e=0,u_e \equiv 0)\in\mathbb R^n\times H^1(0,1;\mathbb R)$. The main result of this article is the construction of numerically tractable sufficient conditions, to obtain the exponential stability around the steady state $(0,0)$, which definition is recalled:
\begin{definition}
 System \eqref{ODEheat} is said to be exponentially stable if for all initial conditions $(X^0,u^0) \in\mathbb R^n\times H^1(0,1)$, there exist $K>0$ and $\delta>0$ such that for all $t>0$,
  \begin{equation}\label{ExpStab}
			E(X(t),u(t)) 
			\leq K e^{-\delta t} \left(|X^0|_n^2 + \|u^0\|_{H^1(0,1)}^2\right).
	\end{equation}
\end{definition}
\vspace{-0.5cm}
More precisely, our goal is then to construct a Lyapunov functional in order to narrow the proof of the stability of the complete infinite dimensional system \eqref{ODEheat} to the resolution of linear matrix inequalities (LMI). 
 \section{Main tools}\label{Tools}
Before stating our main result in the next section, we need to give precise details about the technical tools we will use in the proof : a Lyapunov functional, some Legendre polynomials and the Bessel inequality.
 \subsection{Lyapunov functional}\label{Lyapunov}
 
Inspired by the complete Lyapunov-Krasovskii functional, which is a necessary and sufficient conditions for stability for delay systems \cite{Gu03}, we consider a Lyapunov functional candidate for system \eqref{ODEheat} of the form:
\begin{eqnarray*}
	V(X(t),u(t)) = X^\top (t)PX(t)+2X^\top (t)\ds\int_0^1\hspace{-0.2cm}\mathcal Q(x)u(x,t)dx\\
	+\ds\int_0^1\hspace{-0.1cm} \int_0^1 \hspace{-0.1cm} u^\top (x_1,t)\mathcal T(x_1,x_2)u(x_2,t)dx_1dx_2\\
	+~\alpha \ds\int_0^1 \hspace{-0.1cm} | u(x,t)|^2  dx 
	+ \beta \ds\int_0^1 \hspace{-0.1cm} | u_x(x,t)|^2  dx,
\end{eqnarray*}
where the matrix $P\in \mathcal S_n^+$ and the functions $\mathcal Q\in C(L^2(0,1;\mathbb R^{n\times m}))$ and $\mathcal T\in C(L^2(0,1;\mathbb S^{m}))$ have to be determined. 
The first term and the two last terms of $V$ are a weighted version of the classical energy $E(t)$ of the system. The term depending on the function $\mathcal T$ has been recently considered in the literature in \cite{Peet-TAC17,AVP-Autom16}. The term depending on $\mathcal Q$ is introduced in order to represent the coupling between the ODE and the heat equation. 
\newline
Our objective is to define this Lyapunov functional in order to reduce the proof of the stability of the complete infinite dimensional system \eqref{ODEheat} to the resolution of LMIs. Since a part of the state $(X,u)$ of the system is distributed ($u$ being the solution of a heat equation and depending on a space variable $x$ in addition to the time $t$), it is proposed to impose a special structure for the functions $\mathcal Q$ and $\mathcal T$ in order to obtain numerically tractable stability conditions. The two functions will actually be build as projection operators over a finite dimensional orthogonal family :  the $N+1$ first shifted Legendre polynomials.

 \subsection{Properties of Legendre Polynomials}\label{Legendre}
 
Let us define here the shifted Legendre polynomials considered over the interval $[0,1]$ and denoted $\{ \mathcal L_k\}_{k\in \mathbb N}$. Instead of giving the explicit formula of these polynomials, we detail here their principal properties. One can find details and proofs in \cite{CouHilb-book}.
To begin with, the family $\{ \mathcal L_k\}_{k\in \mathbb N}$ is known to form an orthogonal basis of $L^2(0,1;\mathbb R)$ since 
$
\left< \mathcal L_j , \mathcal L_k \right> = \int_0^1 \mathcal L_j(x) \mathcal L_k(x)  dx = \frac 1{2k+1}\delta_{jk},
$
where $\delta_{jk}$ denotes the Kronecker delta, equal to $1$ if $j = k$ and to $0$ otherwise. Denote the corresponding norm of this inner scalar product
$
\|\mathcal L_k\| = \sqrt{\left< \mathcal L_k , \mathcal L_k \right>}  = 1/\sqrt{2k+1}.
$
The boundary values are given by:
	\begin{equation}\label{BC}
		\mathcal L_k (0) = (-1)^k , \qquad \mathcal L_k (1) =  1.
	\end{equation}
The first shifted  Legendre polynomials are: $\mathcal L_0(x) = 1$, $\mathcal L_1(x) = 2x-1$, $\mathcal L_2(x) =6x^2-6x +1$. Furthermore, the following derivation formula holds: 
	\begin{equation}\label{Deriv1}
		\mathcal L_k'(x) = \ds\sum_{j=0}^{k-1} (2j+1) (1-(-1)^{k+j})\mathcal L_j(x),	k\geq1,
	\end{equation}
	from which, denoting   $\ell_{kj} = (2j+1)(1-(-1)^{k+j})$  if  $ j\leq  k-1$ and $\ell_{kj} = 0$ if $ j\geq k$
we deduce that  for all	$k\geq2$, $\mathcal L_k''(x) = \ds\sum_{j=1}^{k-1} \sum_{i=0}^{j-1} \ell_{kj} \ell_{ji}\mathcal L_i(x)$, 
and $\mathcal L_0''(x) =\mathcal L_1''(x) =0$.
\begin{Rk}
For the record, the classical Legendre polynomials are defined on $[-1,1]$ as the orthonormalization of the family $\{1,x,x^2,x^3,... \}$ but are shifted here to $[0,1]$.
\end{Rk}
\vspace{-0.1cm}
It is now important to notice that any $y \in L^2(0,1)$ can be written
$
y = \ds\sum_{k\geq0} \left< y ,\mathcal L_k
\right> \mathcal L_k/ \|\mathcal L_k\|^2
$
and to set here 
\begin{equation}\label{def_mat1}
\begin{array}{rcll}
	U_N(t)& =&  \ds\vect_{k=0..N}  \left< u(t), \mathcal L_k \right>\quad& \mbox{in } \mathbb R^{N+1}, \\
	\mathds{1}_N &= & \left[ \begin{matrix} 1 & 1&\dots & 1 \end{matrix}\right]^\top  
		\quad& \mbox{in } \mathbb R^{N+1}, \\
	\mathds{1}_N^* &= & \left[ \begin{matrix} 1 & -1 &\dots & (-1)^N  \end{matrix}\right]^\top  
		\quad& \mbox{in } \mathbb R^{N+1},  \\
	L_N &=	&(\ell_{ij})_{i,j = 0..N} 
		\quad &\mbox{in } \mathbb R^{N+1, N+1},\\
	\mathcal{I}_{N} &= & diag (1 , 3,\dots, 2N+1 ) 
		\quad& \mbox{in } \mathbb R^{N+1,N+1}.  \\
\end{array}
\end{equation}
One should notice that for all $N\in \mathbb N^*$, the $L_N$ matrices are strictly lower triangular thanks to the definition of the $\ell_{k,j}$ below \eqref{Deriv1}. The following notations, that we will use below, stems from this:
\begin{equation}\label{def_mat2}
\begin{array}{c}
	L_N = \begin{bmatrix}
			L_{1,N}~  {\bf 0}_{N+1,1}
			\end{bmatrix}
		 \mbox{ with } L_{1,N}\mbox{ in } \mathbb R^{N+1, N},\\
	L_N^2 =	\begin{bmatrix}
			L_{2,N} ~ {\bf 0}_{N+1,2}
			\end{bmatrix}
		 \mbox{ with } L_{2,N}\mbox{ in } \mathbb R^{N+1, N-1}.
\end{array}
\end{equation}
The following properties will be useful for the stability analysis hereafter.
\begin{Pp}\label{PropDerivX}
	Let $u\in C(\mathbb R_+;L^2(0,1))$ satisfy the heat equation and its boundary conditions in \eqref{ODEheat}. The following  formula holds:
 		 \begin{eqnarray}	
		&\vect&_{k=0..N}  \left< \partial_x u(t), \mathcal L_k \right> \nonumber\\
		&=& - L_N U_N(t) + \mathds{1}_N u(1,t)  - \mathds{1}_N^* CX(t)   \label{L11}\\
		&=& - L_{1,N} U_{N-1}(t) + \mathds{1}_N u(1,t)  - \mathds{1}_N^* CX(t)  \label{L22}\\
		&=&  \left[ \begin{smallmatrix}
			~ - C^\top \mathds{1}_N^{*\top}  ~\\ 
			~ \mathds{1}_N^\top ~ \\
			~ - L_{1,N} ^\top ~ 
\end{smallmatrix}\right]^\top 
			\left[ \begin{smallmatrix}
			~ X(t) ~\\ 
			~ u(1,t) ~ \\
			~ U_{N-1}(t) ~ \end{smallmatrix}\right].\nonumber
		\end{eqnarray}
\end{Pp}
 \vspace{-0.7cm}
\begin{proof}
An integration by parts and the first derivation formula \eqref{Deriv1} of the Legendre polynomials yield
\begin{eqnarray*}
		 \left< \partial_x u(t) , \mathcal L_0 \right> &=& u(1,t) - u(0,t), \quad \mbox{ and } \forall k\geq 1\\
		\left< \partial_x u(t) , \mathcal L_k \right> &=&- \ds\sum_{j=0}^{k-1}  \ell_{kj} \! \left< u(t), \mathcal L_j \right> \!+\! u(1,t)  \!-\! u(0,t) (-1)^k.
 \end{eqnarray*}

Using the notations introduced in \eqref{def_mat1} we obtain equation \eqref{L11} and one can deduce \eqref{L22} from \eqref{def_mat2}.
\end{proof}
\begin{Pp}\label{PropDeriv}
	Let $u\in C(\mathbb R_+;L^2(0,1))$ satisfy the heat equation and its boundary conditions in \eqref{ODEheat}. The following time derivative formula holds if $\partial_t u\in C(\mathbb R_+;L^2(0,1 ))$:
 		 \begin{eqnarray}	
			&\dfrac 1\gamma& \dfrac d{dt} U_N(t) = \dfrac 1\gamma \vect_{k=0..N}  \left< \partial_t u(t), \mathcal L_k \right>\nonumber\\
		&~~=& L_N^2 U_N(t) + L_N \mathds{1}_N^* CX(t) - L_N  \mathds{1}_N u(1,t) \nonumber  \\
		&&- \mathds{1}_N^* u_x(0,t) \label{L1}\nonumber\\
		&~~=& L_{2,N} U_{N-2}(t) + L_N \mathds{1}_N^* CX(t)  - L_N  \mathds{1}_N u(1,t)  \\
		&&- \mathds{1}_N^* u_x(0,t) \nonumber\\
		&~~=& \left[ \begin{smallmatrix}
			~ C^\top  \mathds{1}_N^{*\top} L_N^\top ~\\ 
			~ - \mathds{1}_N^\top L_N^\top   ~ \\
			~ - \mathds{1}_N^{*\top} ~ \\
			~  L_{2,N}^\top ~ 
\end{smallmatrix}\right]^\top 
			\left[ \begin{smallmatrix}
			~ X(t) ~\\ 
			~ u(1,t) ~ \\
			~ u_x(0,t) ~ \\
			~ U_{N-2}(t) ~ \end{smallmatrix}\right]. \label{L2} 
		\end{eqnarray}
\end{Pp}
\begin{proof}
We obtain easily, using the heat equation and integrations by parts, along with the boundary \eqref{BC} and derivation formulas of the Legendre polynomials, that
\begin{eqnarray*}
		&&\dfrac d{dt} \left< u(t) , \mathcal L_0 \right>= -\gamma u_x(0,t), \\
 		&&\dfrac d{dt} \left< u(t) , \mathcal L_1 \right> = 2\gamma u(0,t) -2\gamma u(1,t) +\gamma u_x(0,t),\\
		&&	\dfrac d{dt} \left< u(t) , \mathcal L_k \right> =\gamma \ds\sum_{j=1}^{k-1} \sum_{i=0}^{j-1} \ell_{kj} \ell_{ji}\left< u(t), \mathcal L_i \right>
		+\gamma u(0,t)\times\\
			&&\qquad \ds\sum_{j=0}^{k-1} \ell_{kj} (-1)^j 
			-\gamma u(1,t)\sum_{j=0}^{k-1} \ell_{kj}  -\gamma u_x(0,t) (-1)^k
 \end{eqnarray*}
for all $k\geq 2$. The notations introduced in \eqref{def_mat1} allow to conclude to equation \eqref{L1}. It is then easy to deduce \eqref{L2} from  \eqref{def_mat2}. 
\end{proof}
 \begin{Rk}
 It is important to notice here that the main reason for the choice of a base of polynomials to truncate the infinite dimensional state $u$ is the fact that the derivation matrices $L_N$ and $L_N^2$ are strictly lower triangular. It has interesting consequences on the stability study of the whole system~\eqref{ODEheat} and is the cornerstone to obtain a hierarchy of tractable LMIs, in the same vein as in \cite{SG-SCL15}.
 \end {Rk}

 \subsection{Bessel-Legendre Inequality}\label{Bessel}
The following lemma provides a useful information.
 
\begin{Lem}\label{PropLeg}
	Let $u \in C(\mathbb R_+;L^2(0,1))$. The following  integral inequality holds for all $N\in \mathbb N$:
		\begin{equation}\label{Bessel}
			||u(t)||^2 \geq  U_N(t)^\top \mathcal I_N U_N(t) .
		\end{equation}
\end{Lem}
\begin{proof}
Estimate \eqref{Bessel} can be called the Bessel-Legendre inequality.
Since 
$
	u(t) = \sum_{k\geq0} \left< u(t) , \mathcal L_k \right> \mathcal L_k/\|\mathcal L_k\|^2,
$
using the orthogonality of the Legendre polynomials and 
$
	\|\mathcal L_k\|^2 = \left< \mathcal L_k , \mathcal L_k \right>  = 1/(2k+1),
$
we easily get 
$$
			\int_0^1\!\! u(x,t)^2 dx =   \sum_{k\geq0}\frac{\left< u(t), \mathcal L_k \right>^2 }{\|\mathcal L_k\|^2}
			\geq   \sum_{k=0}^N(2k\!+\!1) \left< u(t), \mathcal L_k \right>^2 .
$$
The formulation of Lemma~\ref{PropLeg} stems from notation \eqref{def_mat1}.
\end{proof}
 \section{Stability Analysis}\label{Stability}
 \subsection{Exponential stability result}\label{}
 Following the previous developments, $N$ being a prescribed positive integer, 
we introduce an approximate state of size $n+N+1$, composed by the state of the ODE system $X$ and the projection of the infinite dimensional state $u$ over the set of the Legendre polynomial of degree less than $N$. In other words, the approximate finite dimensional state vector is given by
$$
\left[\begin{array}{c}
	X(t)\\
	U_N(t)
\end{array}\right] = 
			\left[ \begin{matrix}
				X(t) \\ \ds\vect_{k=0..N}  \left< u(t), \mathcal L_k \right>
			\end{matrix} \right].
$$
The main objective of this article is to provide the following stability result for the coupled system~\eqref{ODEheat}, which is based on an appropriate Lyapunov  functional and the use of Property~\ref{PropDeriv}  and Lemma \ref{PropLeg}.

\begin{thm}\label{ThStab}
	Consider system \eqref{ODEheat} with a given thermal diffusivity $\gamma >0$. 
	If there exists an integer $N\geq 0$, such that there exist $\delta >0$, $\alpha>0$, $P\in \mathcal S_{n}$, $Q\in \mathbb R^{n,(N+1)m}$ and $T\in \mathcal S_{(N+1)m}$ satisfying the following LMIs 
	\begin{equation}\label{LMI2}
			\Phi_N=\left[ \begin{matrix}
			P & Q \\ 
			Q^\top & T
			\end{matrix}\right]
		\succ 0,
	\end{equation}
	\begin{equation}\label{LMI1}
			\Psi_{N} (\gamma) = \widetilde\Psi_{N} - \alpha\gamma  \Psi_{N,2} - 2 \beta \gamma   \Psi_{N,3} 
		\prec 0,
	\end{equation}
where 
	\begin{equation}\label{psi1}
			\widetilde\Psi_{N}=\hspace{-0.1cm}
			\begin{bmatrix}
			\Psi_{11} & PB -\gamma Q L_N \mathds{1}_N  & \hspace{-0.3cm}\Psi_{13}& \Psi_{14} \\ 
			{\ast} & 0 &\hspace{-0.3cm} - \beta B^\top C^\top &\Psi_{24}\\ 
			{\ast} &{\ast}  &\hspace{-0.3cm} 0 & -\gamma \mathds{1}_N^{*\top}T  \\
			{\ast} &{\ast}  &\hspace{-0.3cm}{\ast}  & \Psi_{44} \end{bmatrix}
	\end{equation}
with $\Psi_{11} = \mathrm{He}(PA + \gamma Q L_N \mathds{1}_N^{*} C )$, 
$\Psi_{13}=-\gamma Q  \mathds{1}_N^*  - \alpha\gamma C^\top - \beta A^\top C^\top$, 
$\Psi_{14}=A^\top Q + \gamma C^\top \mathds{1}_N^{*\top} L_N^\top T + \gamma Q L_N^2 $,
$\Psi_{24}=B^\top Q -\gamma \mathds{1}_N^{T} L_N^\top T$, and 
$\Psi_{44} =  \mathrm{He}(\gamma L_N^{2\top} T)$,
\begin{equation}\label{psi2}	\hspace{-0.8cm}
	\begin{array}{rcl}
	\Psi_{N,2}&=&  \left[ \begin{smallmatrix}
			 - C^\top \mathds{1}_{N+1}^{*\top} \\ 
			 \mathds{1}_{N+1}^\top \\
			{\bf 0}_{1, N+2} \\
			- L_{1,N+1}^\top \end{smallmatrix}\right]
			 \mathcal I_{N+1} 
			 \left[ \begin{smallmatrix}
			 - C^\top \mathds{1}_{N+1}^{*\top} \\ 
			 \mathds{1}_{N+1}^\top \\
			{\bf 0}_{1, N+2} \\
			- L_{1,N+1}^\top \end{smallmatrix}\right]^\top \\
        		  \Psi_{N,3}&=&  \left[\begin{smallmatrix}
			 C^\top \mathds{1}_{N+2}^{*\top} L_{N+2}^\top \\ 
			 -   \mathds{1}_{N+2}^\top L_{N+2}^\top \\
			 - \mathds{1}_{N+2}^{*\top}  \\
			  L_{2,N+2}^\top ~ \end{smallmatrix}\right]
			 \mathcal I_{N+2} 
			  \left[\begin{smallmatrix}
			 C^\top \mathds{1}_{N+2}^{*\top} L_{N+2}^\top \\ 
			 -   \mathds{1}_{N+2}^\top L_{N+2}^\top \\
			 - \mathds{1}_{N+2}^{*\top}  \\
			  L_{2,N+2}^\top ~ \end{smallmatrix}\right]^\top 
			\end{array}
	\end{equation}
 then the coupled system~\eqref{ODEheat} is exponentially stable. 
Indeed, there exist constants $K>0$ and $\delta> 0$ such that:
\begin{equation}\label{ExpStab}
			E(t) 
			\leq K e^{-\delta t} \left(|X^0|_n^2 + \|u^0\|^2\right),  \forall t>0.
	\end{equation}
\end{thm}
 
 \begin{Rk}
One can point out the robustness of the approach with respect to the triplet $(A,B,\gamma)$, meaning that we could have  $A,B$ and $\gamma$ uncertain, switched, or time-varying... without loosing the stability property. It suffices indeed then to test these LMIs at the vertices of a polytope defining the uncertainties of the triplet. 
\end{Rk}

 In order to reveal the approximate state $U_N$ in the candidate Lyapunov functional $V$ written in section~\ref{Lyapunov}, we select the functions $\mathcal Q$ and $\mathcal T$ as follows:
$
\mathcal Q(x)=\sum_{k=0}^NQ_k\mathcal L_k(x)$, where $\{Q_i\}_{i = 0..N}$ belong to $ \mathbb R^{n}$ and 
$\mathcal T(x_1,x_2)=\sum_{i=0}^N\sum_{j=0}^NT_{ij}\mathcal L_i(x_1)\mathcal L_j(x_2)$,
where $\{T_{ij} =T_{ji}^\top \}_{i,j = 0..N}$ belong to $ \mathbb R$. 
Therefore we can write
\begin{eqnarray}\label{VNdef}
	&&V_N(t) :=V(X(t),u(t))=  
\begin{bmatrix}	X(t)\\U_N(t) \end{bmatrix}^\top 
\begin{bmatrix}	P & Q \\ Q^\top & T	\end{bmatrix}
\begin{bmatrix}	X(t)\\U_N(t) \end{bmatrix}\nonumber\\
&&+~ \alpha  \ds\int_0^1 \hspace{-0.1cm} |u(x,t)|^2  dx
+ \beta  \ds\int_0^1 \hspace{-0.1cm} |\partial_x u(x,t)|^2  dx, 
\end{eqnarray}
where $Q =  [Q_0\  \dots \ Q_N] \in \mathbb R^{n,N+1}$ and  
$T =	(T_{jk})_{j,k = 0..N}  $ in  $\mathbb R^{N+1,N+1}$.
In the following subsection, conditions for exponential stability of the origin of system \eqref{ODEheat} can be obtained using the LMI framework. More particularly, we aim at proving that the functional $V_N$ is positive definite and satisfies $\dot V_N(t) + 2 \delta V_N(t) \leq 0$ for a prescribed $\delta > 0$ and under the LMIs of Theorem~\ref{ThStab}.
 
 \subsection{Proof of the Stability Theorem}\label{proof}
The proof consists in showing that, if the LMIs \eqref{LMI2} and \eqref{LMI1} are verified for a given $N\geq0$, then there exist three positive scalars $\varepsilon_1,\varepsilon_2$ and  $\varepsilon_3$ such that for all $t>0$,
\begin{gather}
 	\varepsilon_1 E(t) \leq V_N(t) \leq   \varepsilon_2 E(t),\label{boundsVN}\\
	 	\dot V_N(t) \leq  -\varepsilon_3 E(t)\label{boundsdVN}.
\end{gather}
Indeed, on the one hand, its suffices to notice that we obtain directly from \eqref{boundsVN} and \eqref{boundsdVN}
$
\dot V_N(t)  + \dfrac{\varepsilon_3}{\varepsilon_2} V_N(t) \leq 0
$
so that 
$
\dfrac d{dt} \left( V_N(t) e^{\varepsilon_3 t / \varepsilon_2 } \right) \leq 0
$
and integrating in time, we get 
$
V_N(t) \leq V_N(0) e^{-\varepsilon_3 t / \varepsilon_2}$
for all $t\geq 0$.
On the other hand, from \eqref{boundsVN}, we can finally write 
$$
\varepsilon_1 E(t) \leq V_N(t)  \leq V_N(0) e^{-\varepsilon_3 t / \varepsilon_2} \leq  \varepsilon_2 E(0)  e^{-\varepsilon_3 t / \varepsilon_2},
$$
allowing to conclude \eqref{ExpStab}.

\textbf{Existence of $\varepsilon _1$:} Since $\alpha>0$, $\beta>0$ and $\Phi_N\succ0$, there exists a sufficiently small $\varepsilon_1 >0$ such that
$ \varepsilon_1 \leq \alpha$, $\varepsilon_1 \leq \beta$  and $\Phi_N= \left[ \begin{smallmatrix}
	P & Q \\ 
	Q^\top & T
\end{smallmatrix}\right]\succ  
\varepsilon_1 \left[ \begin{smallmatrix}
 I_n & 0 \\ 
	0 &0
\end{smallmatrix}\right].$ 
Therefore, we obtain a lower bound of $V_N$ depending on the energy $E(t)$:
$$
\begin{array}{lcl}
V_N(t)
&\geq& \varepsilon_1 (|X(t)|_n^2+ \|u(t)\|^2)+\beta ||\partial_xu(t)||^2\geq\varepsilon_1 E(t).\\
\end{array}
$$
 \vspace{-0.7cm}

\textbf{Existence of $\varepsilon _2$:} There exists a sufficiently large scalar $\lambda>0$ such that 
$
 \left[ \begin{smallmatrix}
	P & Q \\ 
	Q^\top & T
\end{smallmatrix}\right]\preceq \lambda  \left[ \begin{smallmatrix}
	I_n & 0 \\ 
	0 & \mathcal I_N
\end{smallmatrix}\right],
$ 
yielding 
$$	V_N \leq \lambda |X|_n^2 +\lambda U_N^\top \mathcal I_NU_N + \alpha \|u\|^2 + \beta \|\partial_x u\|^2 .
$$
Applying Lemma~\ref{PropLeg} to the second term of the right-hand side ensures that
with $\varepsilon_2 = \max \{\lambda_{\max}\left[\begin{smallmatrix}
	P & Q \\ 
	Q^\top & T
\end{smallmatrix} \right]+\alpha,\beta \}$, one has 
$
	V_N(t) 
	\leq \varepsilon_2 E(t).
$

\textbf{Existence of $\varepsilon _3$:}
In order to prove now that \eqref{boundsdVN} relies on the solvability of the LMI \eqref{LMI1}, 
we need to define an augmented approximate state vector of size $n +N+3$ given by
$
	\xi_N(t)  = \left[ X(t)^\top, u(1,t), u_x(0,t), U_N(t)^\top \right]^\top
$. For simplicity, we omit the variable $t$ in the sequel.

\textit{Step 1:}
Let us split the computation of $\dot V_N$ into three terms, namely $\dot V_{N,1}$, $\dot V_{N,2}$ and $\dot V_{N,3}$ corresponding to each term of $V_N$ in \eqref{VNdef}. On the one hand, using the first equation in system \eqref{ODEheat} and Property~\ref{PropDeriv}, we have :
$$
\dfrac d{dt}  \left[\begin{smallmatrix}
	X\\
	U_N
\end{smallmatrix} \right]= \left[\begin{smallmatrix}
	AX+ Bu(1)\\ 
	\gamma L_N^2 U_N +\gamma L_N \mathds{1}_N^* CX  -\gamma L_N  \mathds{1}_N u(1)   -\gamma \mathds{1}_N^* u_x(0) 
			\end{smallmatrix}\right]
$$
so that we can calculate
$$
\dot V_{N,1} 
 =  \dfrac d{dt} \left( \left[\begin{smallmatrix}
	X\\
	U_N
\end{smallmatrix}\right] ^\top \left[ \begin{smallmatrix}
			P & Q \\ 
			Q^\top & T
			\end{smallmatrix}\right]
 \left[\begin{smallmatrix}
	X\\
	U_N
\end{smallmatrix}\right]  \right) 
= \xi_N^\top ~\Psi_{N,1}(\gamma)~   \xi_N
$$
with 
	$		\Psi_{N,1}=
			 \left[ \begin{smallmatrix}
			~\Psi_{11} & PB -\gamma Q L_N \mathds{1}_N  & -\gamma Q  \mathds{1}_N^* & \Psi_{14} \\ 
			{\ast} & 0 & 0 &\Psi_{24}\\ 
			{\ast} &{\ast}  & 0 & -\gamma \mathds{1}_N^{*\top}T  \\
			{\ast} &{\ast}  &{\ast}  & \Psi_{44} \end{smallmatrix}\right] $
where $\Psi_{11} $, $\Psi_{14} $, $\Psi_{24}$ and $\Psi_{44} $ are defined in Theorem~\ref{ThStab}.

On the other hand, using the heat equation in \eqref{ODEheat}, and an integration by parts, we get both
 \begin{eqnarray*}
 	 \dot V_{N,2}& =&\alpha  \int_0^1 \hspace{-0.1cm} \partial_t\left(  \left| u(x)\right|^2  \right) dx  
	 =2 \alpha  \int_0^1 \hspace{-0.1cm} u(x) \partial_t u(x) dx \\
	&=&  2 \alpha \gamma  \int_0^1  u(x) \partial_{xx}u(x) dx \\
	&=& - 2 \alpha \gamma  \int_0^1  \left| \partial_{x}u(x)\right|^2 dx  + 2 \alpha \gamma  \left[  u \partial_{x}u \right]_{0}^1 \\
	&=&  - 2 \alpha \gamma  \|\partial_{x}u\|^2  - 2 \alpha \gamma C Xu_x(0) 
\end{eqnarray*}
and 
\begin{eqnarray*}
 	 \dot V_{N,3}& =&\beta  \int_0^1 \hspace{-0.1cm} \partial_t\left(  \left| \partial_xu(x)\right|^2  \right) dx 
	 = 2 \beta  \int_0^1 \hspace{-0.1cm} \partial_{tx}u(x) \partial_xu(x)  dx  \\
	&=&  - 2 \beta   \int_0^1  \partial_t u(x) \partial_{xx}u(x) dx + 2 \beta  \left[ \partial_{t} u \partial_{x}u \right]_{0}^1\\
	&=& - 2 \dfrac\beta \gamma  \int_0^1  \left| \partial_{t}u(x)\right|^2 dx  - 2 \beta \partial_{t} u(0) \partial_{x}u(0) \\
	&=& - 2 \dfrac\beta \gamma   \|\partial_{t}u\|^2  - 2 \beta \partial_{x}u(0) C(AX +Bu(1)).
\end{eqnarray*}
Merging the expressions of $\dot V_{N,1},\dot V_{N,2}$ and $\dot V_{N,3}$ yields
\begin{eqnarray}\label{VNdot}		
	\dot V_N &=&  \xi_N^\top ~\Psi_{N,1}(\gamma)~   \xi_N  
	- 2 \alpha \gamma \|\partial_{x}u\|^2 - 2 \dfrac\beta \gamma \|\partial_t u\|^2 \nonumber \\
	&& - 2 \alpha \gamma C Xu_x(0)   - 2 \beta \partial_{x}u(0) C(AX +Bu(1)) \nonumber\\
	&=&  \xi_N^\top ~\widetilde\Psi_{N}(\gamma)~   \xi_N  
	- 2 \alpha \gamma \|\partial_{x}u\|^2 - 2 \dfrac\beta \gamma \|\partial_t u\|^2 
\end{eqnarray}
where $\widetilde\Psi_{N}(\gamma)$ is defined in \eqref{psi1}.

\textit{Step 2:}
Let us explain here how we can deal with the terms $ \|\partial_{x}u(t)\|^2$ and $ \|\partial_{t}u(t)\|^2$. Following the proof of Lemma~\ref{PropLeg}, up to the order $N+1$,  we can write, using an integration by parts and the derivation formula in Property~\ref{PropDerivX} of the Legendre polynomial 
 \begin{eqnarray*}
 	&& \|\partial_{x}u\|^2
	 	\geq   \sum_{k=0}^{N+1}(2k+1)\left| \left<\partial_{x} u, \mathcal L_k \right>\right|^2\\
		 &\geq& \left[ \begin{smallmatrix}
			~ X ~\\ 
			~ u(1) ~ \\
			~ U_{N} ~ \end{smallmatrix}\right]^\top 
			\left[ \begin{smallmatrix}
			~ -  C^\top \mathds{1}_{N+1}^{*\top} ~\\ 
			~ \mathds{1}_{N+1}^\top ~ \\
			~ - L_{1,N+1}^\top ~ \end{smallmatrix}\right]
			 \mathcal I_{N+1} 
			 \left[ \begin{smallmatrix}
			~ - C^\top \mathds{1}_{N+1}^{*\top} ~\\ 
			~ \mathds{1}_{N+1}^\top ~ \\
			~ - L_{1,N+1}^\top ~ \end{smallmatrix}\right]^\top 
			\left[ \begin{smallmatrix}
			~ X ~\\ 
			~ u(1) ~ \\
			~ U_{N} ~ \end{smallmatrix}\right]\\
\end{eqnarray*}
One can deduce that with $\Psi_{N,2} $ defined in \eqref{psi2},
\begin{equation}\label{p2}
 	-  \|\partial_{x}u(t)\|^2 \leq  -   \xi_N^\top (t)~ \Psi_{N,2} ~ \xi_N (t).
\end{equation}
Similarly, using  Property~\ref{PropDeriv} and Lemma~\ref{PropLeg} up to the order $N+2$, we have 
 \begin{multline*}
 	  \dfrac  1\gamma  \|\partial_{t}u(t)\|^2
	 	\geq  \dfrac  1\gamma \dfrac {dU_{N+2}^\top }{dt} ~ \mathcal I_{N+2} \dfrac {dU_{N+2}}{dt}\\
		 \geq \gamma
			\left[ \begin{smallmatrix}
			~ X ~\\ 
			~ u(1) ~ \\
			~ u_x(0) ~ \\
			~ U_{N} ~ \end{smallmatrix}\right]^\top 
			\Psi_{N,3} 
			\left[ \begin{smallmatrix}
			~ X ~\\ 
			~ u(1) ~ \\
			~ u_x(0) ~ \\
			~ U_{N} ~ \end{smallmatrix}\right]
\end{multline*}
with $\Psi_{N,3} $ defined in \eqref{psi2} so that 
\begin{equation}\label{p3}
 	-  \dfrac  1\gamma  \|\partial_{t}u(t)\|^2 \leq - \gamma  \xi_N^\top (t) \Psi_{N,3} ~  \xi_N(t)  .
\end{equation}
\textit{Step 3:} Since we assume $\Psi_N \prec 0$, then choosing $\varepsilon =  \lambda_{\min}(-\Psi_N) /2 $, we get
$\Psi_N \prec  
 - \varepsilon \left[ \begin{smallmatrix}
 I_n & 0 & 0 & 0 \\ 
	* &2 & 0 &  0\\
		* & * & 0 &  0\\
	* & * & * &  0\\
\end{smallmatrix}\right].
$
Therefore, we can write from \eqref{VNdot}, \eqref{p2} and \eqref{p3}, choosing  $\varepsilon_3 = \min \left\{\frac 13 \alpha \gamma, \lambda_{\min}(-\Psi_N) /2 \right\}$,
\begin{eqnarray*}
	\dot V_N(t) & \leq &   \xi_N^\top (t) ~\widetilde\Psi_{N}~   \xi_N  (t) - \alpha \gamma  \|\partial_{x}u(t)\|^2 - 2 \dfrac\beta \gamma \|\partial_t u(t)\|^2 \\
	&&  -  3 \varepsilon_3  \|\partial_{x}u(t)\|^2  \\
	& \leq &   \xi_N^\top (t) \left( \widetilde\Psi_{N} -  \alpha \gamma \Psi_{N,2} - 2 \beta \gamma \Psi_{N,3} \right)   \xi_N  (t)\\
	&&-  3 \varepsilon_3  \|\partial_{x}u(t)\|^2\\
	& \leq &   \xi_N^\top (t)~ \Psi_{N} ~  \xi_N  (t)-  3 \varepsilon_3  \|\partial_{x}u(t)\|^2\\
	& \leq &  - \varepsilon_3 |X(t)|^2_n - 2 \varepsilon_3 |u(1)|^2-  3 \varepsilon_3  \|\partial_{x}u(t)\|^2.
\end{eqnarray*}
Finally, since one can easily prove that for any $u \in H^1(0,1)$,
$$\|u\|^2 \leq 2 |u(1)|^2 + 2  \|\partial_{x}u\|^2,$$ we obtain
$$\dot V_N(t) \leq  - \varepsilon_3 |X(t)|^2_n - \varepsilon_3 \|u(t)\|^2-   \varepsilon_3  \|\partial_{x}u(t)\|^2$$
which is precisely \eqref{boundsdVN}. One can therefore conclude to the exponential stability of system \eqref{ODEheat}.

\section{Numerical example}\label{Numerics}

 Our goal here is to propose a numerical illustration that can highlight the possibilities and tractability of the stability LMI tests provided by Theorem~1. Hence, we are presenting an example where the closed-loop system depends on two parameters : the thermal diffusivity $\gamma$ of the heat equation and a parameter $K$. This numerical example is formerly issued from the field of time delay systems (see e.g. \cite{Sipahi2011}) and $K$ enters the model as follows: 
$$
A=
\left[\begin{smallmatrix}
 0 &0& 1& 0\\
 0& 0& 0& 1\\
 -10 -  K& 10& 0& 0\\
 5 &-15& 0 &-0.25
\end{smallmatrix}\right],\quad 
B=\left[\begin{smallmatrix}
   0\\
   0\\
   1\\
   0
  \end{smallmatrix}\right],\quad C=\left[\begin{smallmatrix}
  K\\ 0\\ 0\\ 0
  \end{smallmatrix}\right]^T.
$$

This data triplet $(A,B,C)$ has indeed already been considered in the context of time delay systems where the delayed matrix is $A_d = BC$. The main motivation for studying this example arises from the fact the stability region has a very complicated shape, that is hard to detect using a Lyapunov-Krasovskii functional  approach. We will see that the stability region is difficult to detect as well for our system \eqref{ODEheat} with these values for $A$, $B$, $C$. 

\begin{figure}[t]
\includegraphics[width=0.5\textwidth]{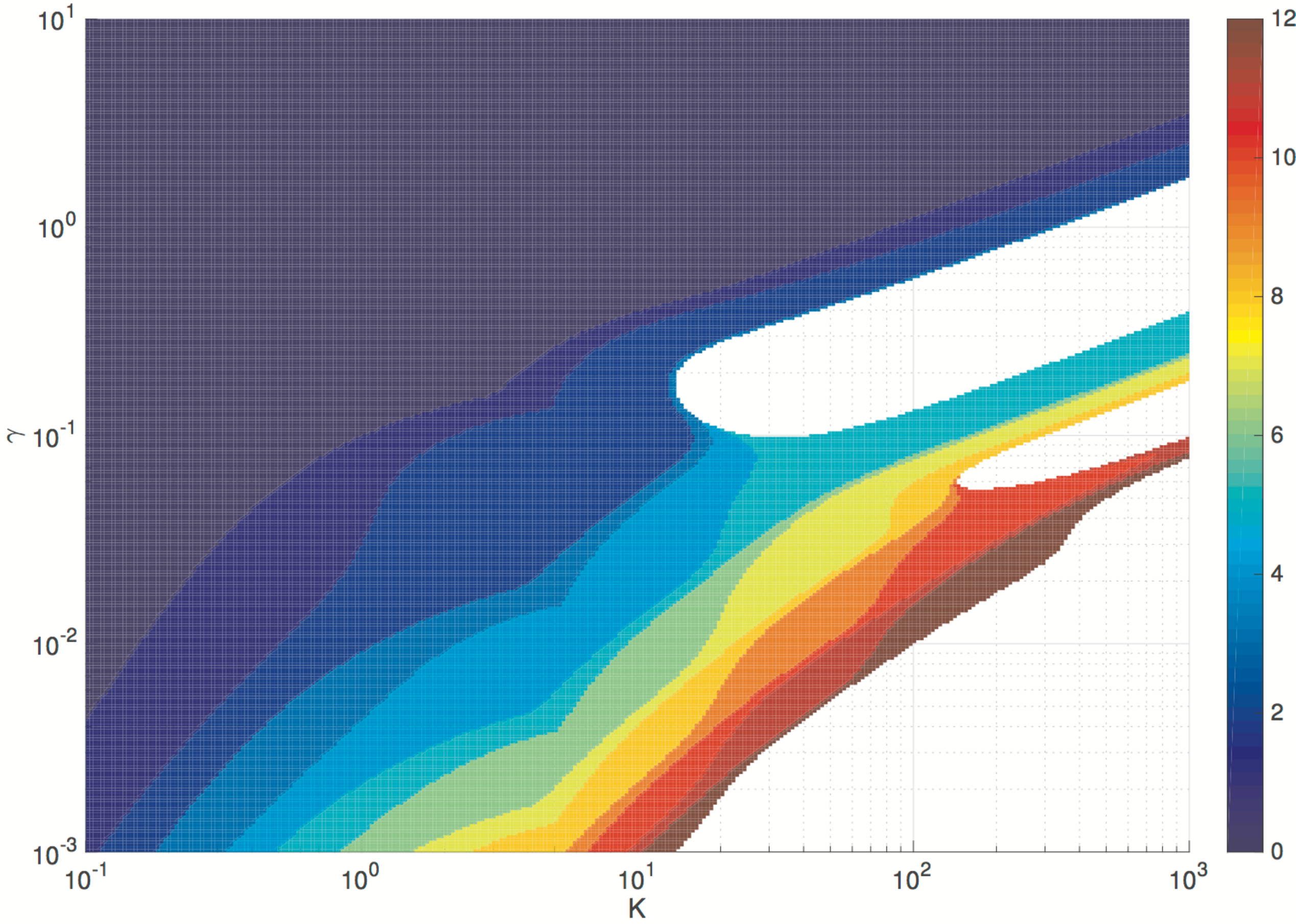}
\caption{Stability region in the plan $(K,\gamma)$, obtained using Theorem \ref{ThStab} for $N=0,\dots, 12$.} \label{Plot_K_gamma}
\end{figure}

In order to illustrate the potentialities of Theorem \ref{ThStab}, we have proposed Figure \ref{Plot_K_gamma}, depicting in the plan $(K,\gamma)$ and in logarithmic scales, for which values of $N$ solutions to the LMI problem (\ref{LMI2}-\ref{LMI1}) have been found.
The white area corresponds to values of $(K,\gamma)$ for which no solutions have been obtained for $N<13$. The darkest area corresponds to the stability region obtained with $N=0$ in Theorem \ref{ThStab}. The general tendency presented in Figure~\ref{Plot_K_gamma} is that for large values of~$\gamma$, stability is guaranteed. However, for small values of $\gamma$, peculiar stability regions are detected. 
One can see that increasing $N$ in Theorem \ref{ThStab} allows to enlarge the stability regions as illustrated in the hierarchical structure of LMIs (\ref{LMI2}-\ref{LMI1}). Interestingly, Figure \ref{Plot_K_gamma} also detects two instability zones, where (\ref{LMI2}) or (\ref{LMI1}) are not solvable, even for larger values of $N$. 
 \begin{Rk}
Figure~\ref{Plot_K_gamma} has also the interest of illustrating the hierarchy that our approach suggests. One sees clearly  the progression of the guaranteed domain of stability with the increase of $N$. 
\end{Rk}
In order to illustrate the stability regions depicted in Figure \ref{Plot_K_gamma}, several temporal simulations of the coupled-system have been provided in Figure \ref{fig:3sim}. They correspond to system \eqref{ODEheat} with the same numerical values $(A,B,C)$ and the particular choice of $K=100$. This selection of $K$ is relevant since there is an interval of values of $\gamma$ included in $[0.1,\ 0.2]$ such that the LMIs conditions of Theorem~\ref{ThStab} are not verified even for large values of~$N$.  Under the initial conditions $ u^0(x)=CX^0 - 20 x(x-2)+10(1-cos(8\pi x))$ and $X^0=\begin{bmatrix}0& 1 &-1 &0\end{bmatrix}$.
and noting that this is compatible with the requirements $u^0(0)=CX^0$ and $\partial _xu^0(1)=0$, 
three simulations are provided with
\begin{itemize}
\item [(a)] $\gamma=1$, corresponding to a stable region according to Theorem~\ref{ThStab} with $N=0$;
\item [(b)]$\gamma=0.2$, corresponding to a region for which Theorem~\ref{ThStab} has no solution for any $N\leq 12$;
\item [(c)]$\gamma=0.05$, which, according to Theorem~\ref{ThStab} with $N\geq 5$, is exponentially stable.  
\end{itemize}
Simulations of the coupled ODE - Heat PDE have been performed using classical tools available in the literature. The ODE has been discretized using a Runge-Kutta algorithm of order 4 with a principal step $\delta_t$. The PDE have been simulated by performing a backward in time central order difference in space with a step $\delta_x$, with $\delta_t\leq \delta_x^2/(2\gamma)$ and $\delta_x=1/20$ to ensure the numerical stability of the approximation. 
\begin{figure}[t]
	\centering
	\subfigure[Simulations results obtained with $\gamma = 1$.]{\includegraphics[width=7cm]{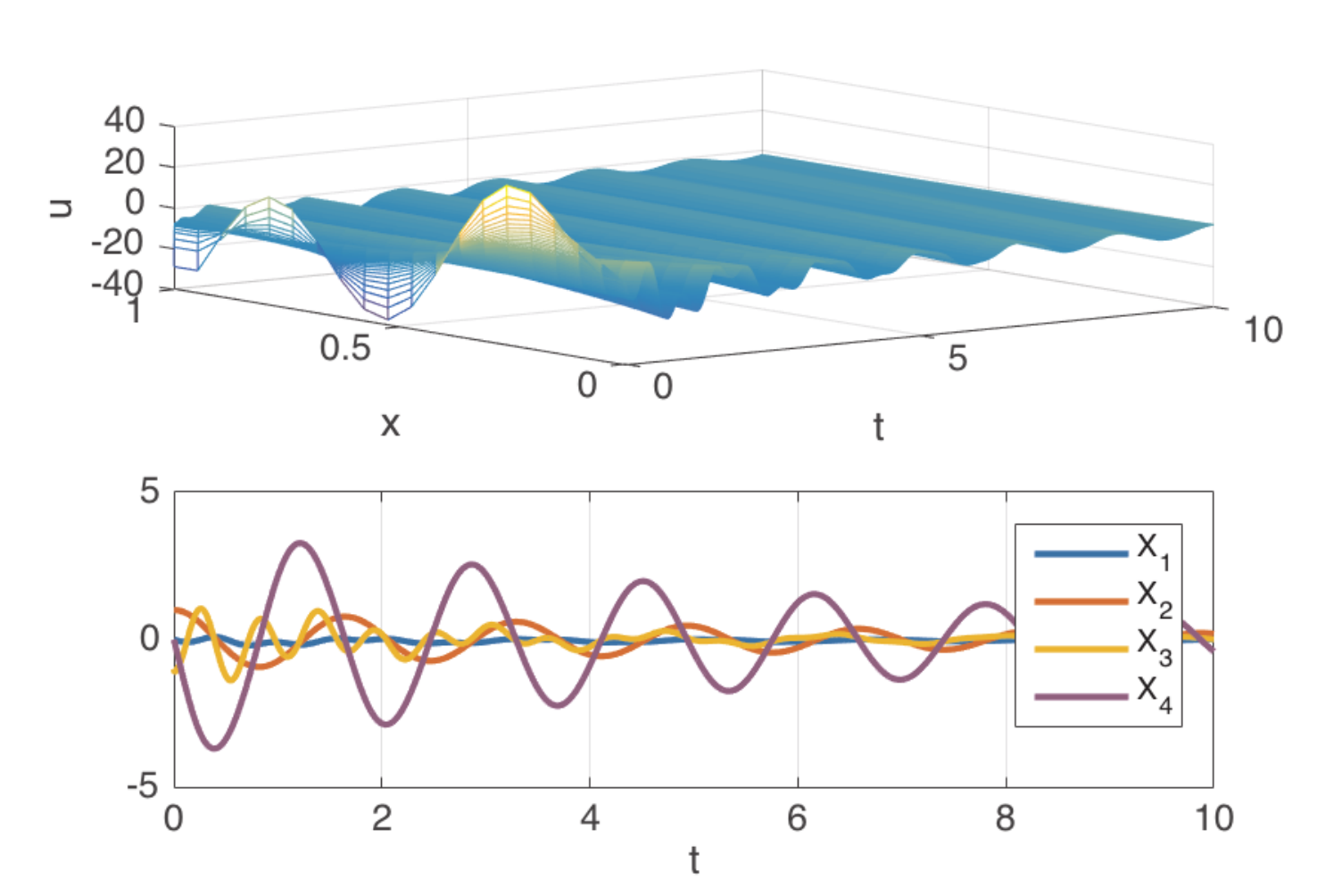}\label{fig:3sim1}}\vspace{0.3cm}
	\subfigure[Simulations results obtained with $\gamma = 0.2$.]{\includegraphics[width=7cm]{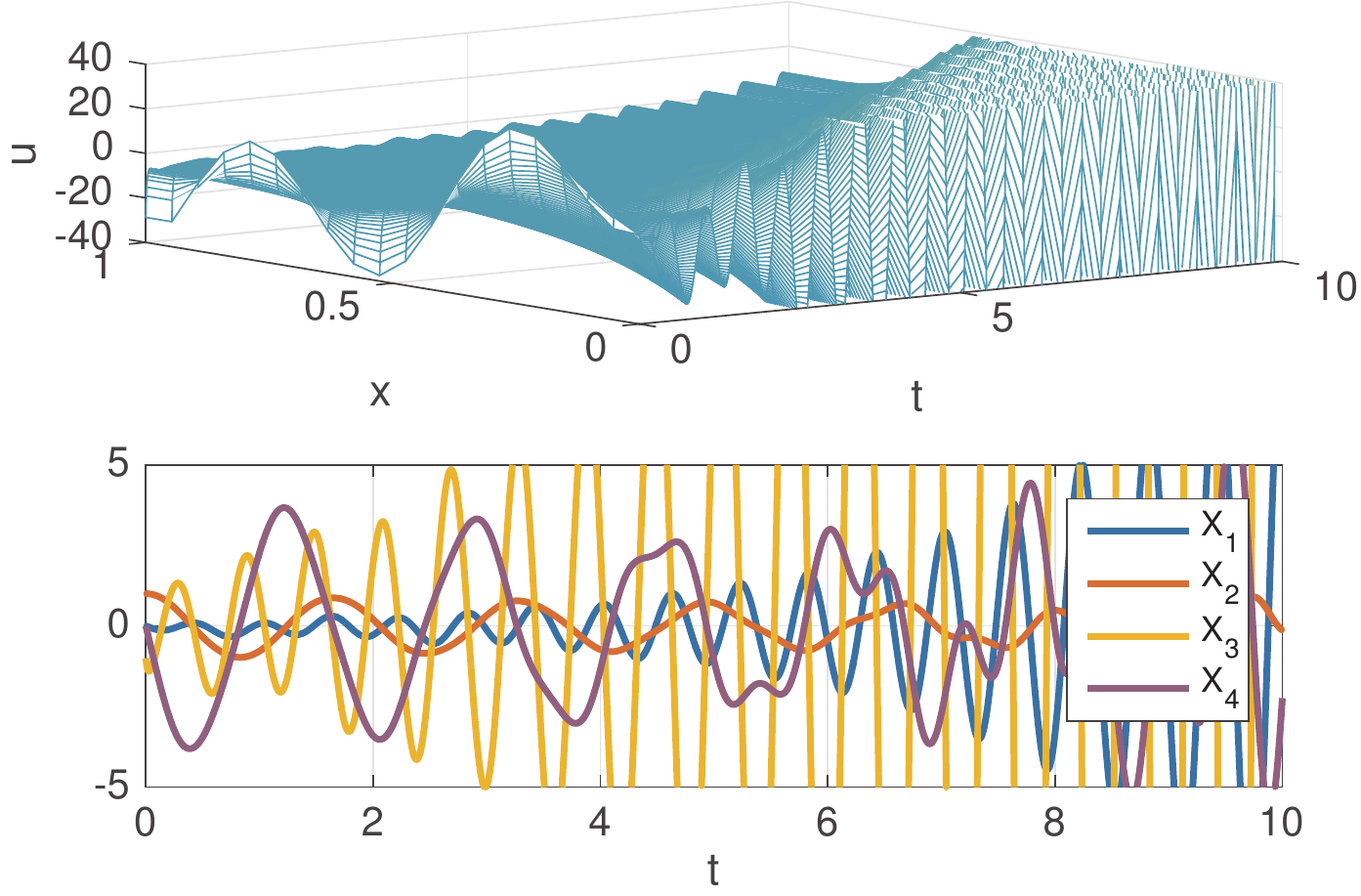}\label{fig:3sim2}}\vspace{0.3cm}
	\subfigure[Simulations results obtained with $\gamma = 0.05$.]{\includegraphics[width=7cm]{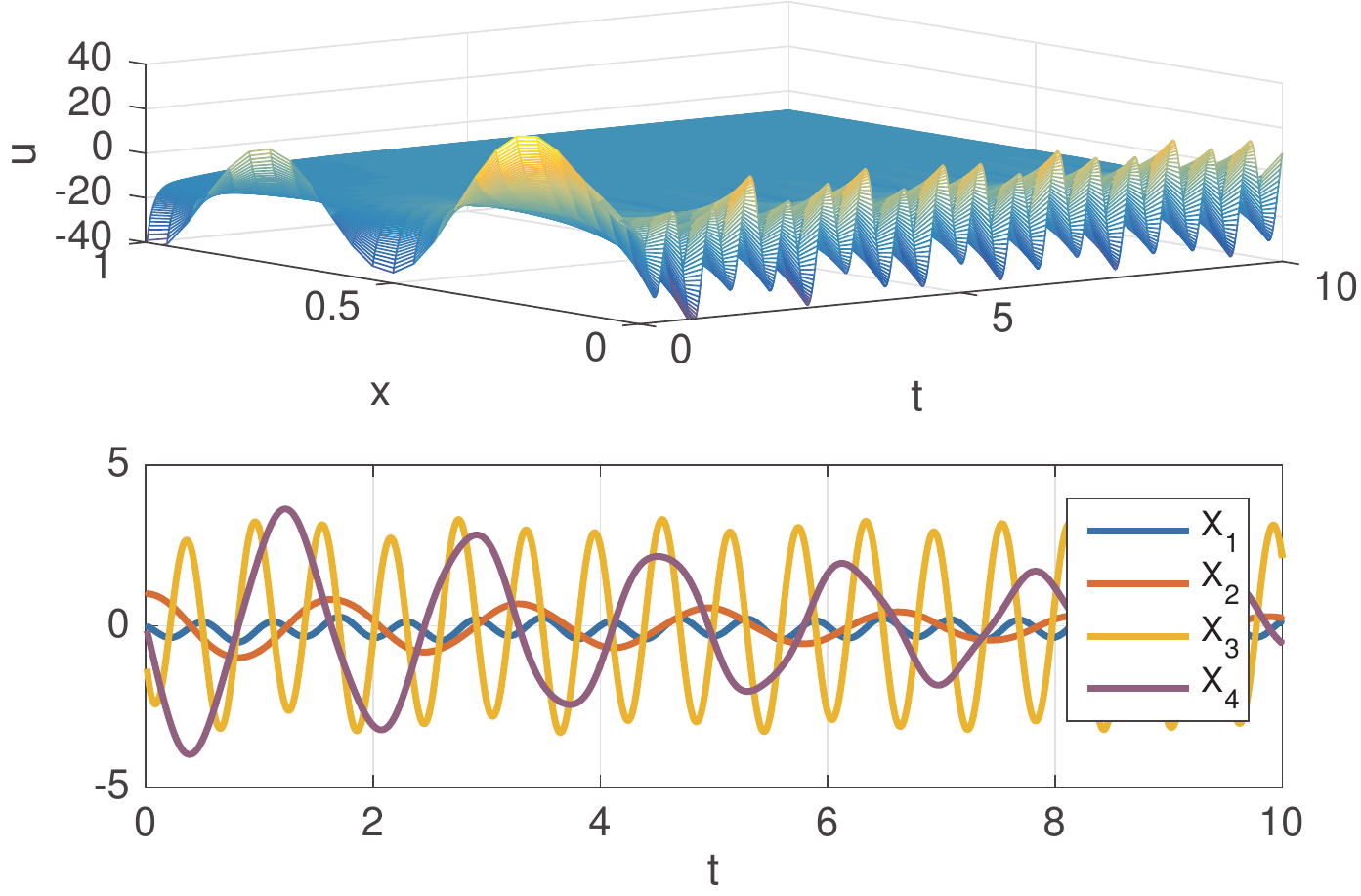}\label{fig:3sim3}}\vspace{0.3cm}
	\vspace{-0.25cm}
	\caption{Evolution of the state $(X,u)$ with respect to time with $K=100$ and for several values of~$\gamma$.}
	\label{fig:3sim}
\end{figure}

Figure \ref{fig:3sim}(a) obviously shows the stable behaviors detected by Theorem~\ref{ThStab} with $N=0$, with a quite fast convergence  to the equilibrium. The illustration of the second case Figure~\ref{fig:3sim}(b) is consistent with Figure~\ref{Plot_K_gamma}, since the solution of this system diverges. This is consistent with the fact that no solutions to the conditions of Theorem~\ref{ThStab} can be found for any $N\leq 12$. More interestingly, the last situation, presented in Figure~\ref{fig:3sim}(c), shows simulations results which are very slowly converging to the origin, with however a lightly damped oscillatory behavior of the state of the ODE and of the PDE close to the boundary $x=0$. On the other side, the state function $u(x,t)$ for sufficiently large values of $x$ is clearly smooth and converges slowly to the origin. Actually, case (c) illustrates a situation where a very small diffusion coefficient $\gamma$ induces a slow convergent behavior for which the conditions of Theorem 1 are only fulfilled for a large parameter $N\geq5$. This may indicate a correlation between the energy decay rate and the smallest $N$ for which the LMIs are verified.

\section{Conclusion and future works}

This article has provided a new and fruitful approach to numerically check the exponential stability of coupled ODE - Heat PDE systems. Our approach relies on the efficient construction of specific Lyapunov functionals allowing to derive diffusion parameter-dependent stability conditions. These tractable conditions of stability are expressed in terms of LMIs and obtained using the Bessel inequality. 
This work is a first contribution in the study of coupled ODE-Heat PDE systems using this framework and has the ambition to provide a method that could prove to be robust and useful in more intricate situations, such as other parabolic PDEs (e.g. in \cite{Day-QAM1983}, \cite{BKL-IEEETAC01}  or reaction-diffusion, Kuramoto-Sivashinski...), or vectorial infinite dimensional state $u$ to handle MIMO systems. A very interesting but challenging question is also the study of the convergence of our result when the order $N$ of truncation grows. We would like to prove that if the stability of the coupled system holds, then there exists an order $N$ for which our LMIs are verified.  Future research will also include  the study of the robustness of our technique with respect to the whole data quadruplet $(A,B,C,\gamma)$. Another possible direction would consist in the inclusion of different and more general formulation of the boundary conditions, which includes Neumann, Robin and Dirichlet type of boundary constraints and coupling conditions.

\bibliographystyle{plain}

\end{document}